\documentclass{amsart}
\usepackage{amsmath}
\usepackage{url}
\usepackage{mathrsfs}
\usepackage{enumitem}   % 允许自定义列表格式，特别是[Step 1:]这种格式
\usepackage{mathtools}  % 扩展了amsmath的功能
\usepackage{amssymb, graphicx, amsrefs}
\usepackage{svg}
\usepackage{booktabs}
\usepackage{siunitx}
\usepackage{color}%cyan,red;magenta,green;yellow,blue
\usepackage{graphicx}
\usepackage{float} 
\usepackage{subfigure}
\usepackage{bbm}
\usepackage{hyperref}
\usepackage{epstopdf}
\usepackage{fancyhdr} 
\usepackage{cases}
\copyrightinfo{2022}{Jin Sun}

\newtheorem{theorem}{Theorem}[]
\newtheorem{lemma}[theorem]{Lemma}

\theoremstyle{definition}
\newtheorem{definition}[theorem]{Definition}

\theoremstyle{remark}
\newtheorem{remark}[theorem]{Remark}

\numberwithin{equation}{section}

\newcommand{\partiald}[2]{  \frac{\partial #1 }{\partial #2}}

\newcommand{\bracket}[1]{\left( #1\right)}

\newcommand{\av}[1]{\left\vert #1\right\vert}
\newcommand{\aV}[1]{\left\Vert #1\right\Vert}

\newcommand{\R}{\mathbb{R}}

\providecommand{\eat}[1]{}

\newcommand{\Hm}[1]{\leavevmode{\marginpar{\tiny%
			$\hbox to 0mm{\hspace*{-0.5mm}$\leftarrow$\hss}%
			\vcenter{\vrule depth 0.1mm height 0.1mm width \the\marginparwidth}%
			\hbox to 0mm{\hss$\rightarrow$\hspace*{-0.5mm}}$\\\relax\raggedright
			#1}}}

\begin{document}
	
    \title[Parabolic Frequency on Gaussian Spaces]{Parabolic Frequency on Gaussian Spaces and Unique Continuation}
    
    \author{Jin Sun}
    \address{Jin Sun, School of Mathematical Sciences, Fudan University, 200433, Shanghai, China}
    \email{jsun22@m.fudan.edu.cn}

    \author{Kui Wang}
    \address{Kui Wang, School of Mathematical Sciences, Soochow University, 215006, Suzhou, China}
    \email{kuiwang@suda.edu.cn}
    
	
	%	\date{\today}
	%	\thanks{The author is supported by Shanghai Science and Technology Program [Project No. 22JC1400100]}
	
\begin{abstract}
We establish an almost-monotonicity formula for a parabolic frequency on Gaussian spaces for solutions of the Ornstein-Uhlenbeck heat equation with lower-order terms:
$$\partial_t u = L_\gamma u + b(x,t) \cdot \nabla u + c(x,t)u, $$
where $L_\gamma = \Delta - x \cdot \nabla$ is the Ornstein-Uhlenbeck operator. In contrast to classical results that require $b$ and $c$ to be bounded, we only assume that $b$ is bounded and $c$ satisfies a linear growth condition, while the solution $u$ is allowed to have at most exponential quadratic growth. The key innovation is a weighted $L^2$ framework that uses the backward Mehler kernel as a weight, which naturally encodes the underlying measure and compensates for the unbounded coefficients. From the frequency monotonicity, we derive the strong unique continuation principle. This extends Poon's seminal results \cite{Poon1996} and complements recent geometric generalizations  by Colding and Minicozzi \cite{ColdingMinicozzi2022} in the context of Gaussian measure spaces. We further apply our framework to establish unique continuation for equations with potentials exhibiting quadratic growth or certain singularities. 
\end{abstract} 
   
\maketitle
	
\section{Introduction}

The frequency function method was introduced by Almgren \cite{Almgren1979} and systematically developed by Garofalo and Lin \cite{GarofaloLin1986} to study the unique continuation property of elliptic equations. Lin \cite{Lin1990} extended these techniques to parabolic equations, establishing a uniqueness theorem for solutions of the heat equation. Subsequently, Poon \cite{Poon1996} established the monotonicity of parabolic frequency for solutions of $\partial_t u - \Delta u = b \cdot \nabla u + cu$ with bounded coefficients, proving that solutions vanishing to infinite order at any space-time point must be identically zero. Carleman estimate techniques provide a complementary approach to unique continuation for parabolic equations with less regular coefficients, as demonstrated in \cite{KochTataru2009} for second-order parabolic equations with unbounded coefficients. Ni \cites{NiKahler2007, Ni2015} developed matrix Li-Yau-Hamilton estimates for Kähler-Ricci flow and studied parabolic frequency monotonicity in connection with Hardy-Pólya-Szegö theorems. Recently, Li and Wang \cite{LiWang2019} obtained almost-monotonicity formulas on compact manifolds with error terms depending on curvature bounds. Extending these techniques to general Ricci flows, Baldauf and Kim \cite{BaldaufKim2023} proved parabolic frequency monotonicity along  Ricci flows, and Li and Zhang \cite{LiZhang2024} established matrix Li-Yau-Hamilton estimates for the heat equation under Ricci flow and proved monotonicity of parabolic frequency under non-negativity assumptions on sectional curvature.

The Ornstein-Uhlenbeck operator $L_\gamma = \Delta - x \cdot \nabla$ and the associated heat equation on Gaussian spaces have been extensively studied from various perspectives, including their spectral properties in $L^p$ spaces with invariant measure \cites{Metafune2001, Metafune2002} and the theory of Ornstein-Uhlenbeck semigroups \cite{LunardiDaPrato2020}. Colding and Minicozzi \cite{ColdingMinicozzi2018} established sharp frequency bounds for the eigenfunctions of the Ornstein-Uhlenbeck operator. In their subsequent work \cite{ColdingMinicozzi2022}, they proved parabolic frequency monotonicity on general Riemannian manifolds with drift Laplacians, which when specialized to Euclidean space with Ornstein-Uhlenbeck drift implies backward uniqueness for the parabolic equation. However, the development of a well-suited frequency function method for the strong unique continuation property of parabolic equations in this setting, particularly for equations with unbounded coefficients, remains largely open. This gap is notable given the success of such methods in Euclidean and geometric settings. The main contribution of this paper is to fill this gap by establishing the monotonicity of parabolic frequency and the strong unique continuation property for heat equations on Gaussian spaces with lower-order terms of linear or quadratic growth and certain singularities.

Here, 
we denote the Gaussian measure on $\mathbb{R}^n$ by
\begin{equation}\label{E:Gaussian_measure}
    d\gamma = \frac{1}{(2\pi)^{n/2}}e^{-{\av{x}^2}/{2}}.
\end{equation}
Assume that $\Omega\subset\R^n$ is a convex domain. We consider the following heat equation in Gaussian space:
\begin{equation}\label{heat_equation}
        \begin{cases}
            \partial_t u= L_\gamma u+b\cdot\nabla u+cu,  &\qquad (x,t)\in \Omega\times(-T,+\infty),\\
            u(x,t)=0,&\qquad x\in \partial\Omega\times(-T,+\infty),
        \end{cases}
    \end{equation}
where $T>0$ is a constant. When $\Omega=\R^n$,  the problem reduces to the Cauchy problem in the whole space (no boundary condition is needed).

More generally, we will consider the following nonlinear parabolic equation:
\begin{equation}\label{nonlinear_equation}
    \begin{cases}
            \partial_t u = L_\gamma u -  F(u,\nabla u),  &\qquad (x,t)\in \Omega\times(-T,+\infty),\\
            u(x,t)=0,&\qquad x\in \partial\Omega\times(-T,+\infty).
        \end{cases}
\end{equation}

\begin{definition}
    We say that a function $f:\Omega\rightarrow \R$ has $G(A,B)$-growth if there exist constants $A,B>0$, such that
    \begin{equation}\label{E:Gaussian_growth}
        \av{f(x)}\leq B e^{A{\av{x}^2}}, \quad x\in \Omega.
    \end{equation}

\end{definition}

Throughout this paper, we assume that the coefficient $b = (b_1, b_2, \ldots, b_n)$ is bounded and $c$ has at most linear growth:
\begin{equation}\label{b_and_c_bounded}
    \sup_{x \in \mathbb{R}^n, t > 0} \left(|b(x,t)|+  \av{\frac{c(x,t)}{1+\av{x}}} \right) \leq  L
\end{equation}
for some constant $L>0$. 
We say that a function $u$ vanishes to order $K>0$ in space-time at $(x_0, t_0)$ if
\begin{equation}\label{E:vanishing_order_K}
    u(x,t) = O\left( (|x - x_0|^2 + |t - t_0|)^K \right)
\end{equation}
for $(x, t)$ near $(x_0, t_0)$, $t < t_0$. And we say that a function $u$ vanishes to infinite order in space-time at $(x_0, t_0)$ if \eqref{E:vanishing_order_K} holds for any integer $K>0$.

To analyze solutions of \eqref{heat_equation}, we introduce the weighted $L^2$ integrals with respect to the backward Mehler kernel and the Gaussian measure. When $b=0$, $c=0$, we have the following theorem, which tells us that the parabolic frequency of solutions of the heat equation with Dirichlet boundary is monotonically increasing.
\begin{theorem}\label{T:theorem_1}
    Assume that $\Omega\subset\R^n$ is a convex domain or $\Omega=\R^n$. Let $u$ be a solution of the heat equation \eqref{heat_equation} with $b=0$, $c=0$. Suppose that for every $t>-T$, $u\in  W^{2,2}\bracket{\Omega;e^{-2A\av{x}^2}d\gamma}$ for some constant $A>0$. Let $T_0=\frac{1}{2}\min\{\log\bracket{1+\frac{1}{8(A+\pi)}},T\}$.  Then  the parabolic frequency $N(\tau)$ given by \eqref{equation_N} increases monotonically for $0<\tau<T_0$.
\end{theorem}

The construction of weighted $L^2$ integrals naturally incorporates the Gaussian structure of the problem and allows us to handle the linear growth of the coefficients. The resulting parabolic frequency satisfies an almost-monotonicity property that is sufficient to derive a strong unique continuation. And compared to \cite{Poon1996}*{Theorem 1.1}, which requires $b,c,u,\nabla u$ to be bounded, in this paper we only require that $b$ is bounded, $c$ has linear growth, $u$ has $G(A,B)$-growth and $\av{\nabla u}\in  W^{1,2}\bracket{\Omega;e^{-2A\av{x}^2}d\gamma}$. That is, we have the following main theorem.

\begin{theorem}\label{main_theorem}
    Let $u$ be a solution of \eqref{heat_equation} on $\Omega$ and $b$ and $c$ satisfy the assumption \eqref{b_and_c_bounded}. Suppose that for every $t>-T$, $u(x,t)$ has $G(A,B(t))$-growth and $\av{\nabla u}\in  W^{1,2}\bracket{\Omega;e^{-2A\av{x}^2}d\gamma}$ for some constant $A>0$ and for some continuous function $B(t)\in C((-T,+\infty))$. If $u$ vanishes to infinite order in space-time at $(x_0,t_0)\in \Omega\times(-T,+\infty)$, then $u$ is identically zero.
\end{theorem}
\begin{remark}
    Note that when we choose $B(t)=e^{B_0(t+T+1)}$, the $G(A,B(t))$-growth of $u$ means
    \begin{equation*}
        \av{u(x,t)}\leq  e^{A{\av{x}^2}+B_0(t+T+1)}, \quad x\in \Omega,\ t\in(-T,+\infty).
    \end{equation*}
\end{remark}

Our framework is sufficiently robust to handle equations with time-independent potentials $V(x)$ exhibiting quadratic growth or certain singularities. Specifically, we establish the strong unique continuation for
\begin{equation}\label{heat_equation_potential}
        \begin{cases}
            \partial_t u(x,t)= L_\gamma u(x,t)-V(x)u(x,t),  &\qquad (x,t)\in \Omega\times(-T,+\infty),\\
            u(x,t)=0,&\qquad x\in \partial\Omega\times(-T,+\infty),
        \end{cases}
    \end{equation}
where $V(x)$ may grow like $O(|x|^2)$ or possess a singularity of order up to $|x|^{-2}$ at the origin, which is more general than \cite{Poon1996}*{Theorem 1.2}:
\begin{theorem}\label{T:main_theorem_V}
    Let $u$ be a solution of \eqref{heat_equation_potential} on $\Omega$ and $0\in\Omega$. Assume that $w\in C^1(\mathbb{S}^{n-1})$ with $\big\|w\big\|_{L^\infty(\mathbb{S}^{n-1})}\leq L$ for some $L>0$, and that $V(x)=v(\av{x})w(\frac{x}{\av{x}})$ satisfies one of the following conditions.
    ~\\
    (i) For any point $x\in\Omega$, $v(r)=v(\av{x})$ has no singular points and
    \begin{equation}\label{E:V_case_1}
        \av{rv^\prime(r)+2v(r)}\leq {1+r^2};
    \end{equation}
    (ii) $n\geq 3$, $v(r)=r^{-q}$ for some $q\in (0,2]$ and $w\big|_{\mathbb{S}^{n-1}}\geq 0$ when $q<2$.
    
    Suppose  there exist a positive constant $A>0$ and a continuous function $B(t)\in C((-T,+\infty))$ such that for every $t>-T$, $u(x,t)$ has $G(A,B(t))$-growth and $\av{\nabla u}\in  W^{1,2}\bracket{\Omega;e^{-2A\av{x}^2}d\gamma}$. If $u$ vanishes to infinite order in space-time at $(x_0,t_0)\in \Omega\times(-T,+\infty)$, then $u$ is identically zero.
    \end{theorem}
\begin{remark}
    For the first case, $v(r)$ can be chosen to be 
    \begin{equation*}
        v(r) = c_0+\sum_{j=1}^{m}c_jr^{\alpha_j}+cr^2,
    \end{equation*}
    where $c,c_0,\ldots,c_{m}$ are constants and $0<\alpha_j<2$ for each $1\leq j\leq m$. Thus, $V$ is permitted to have quadratic growth, which matches the definition of $G(A,B)$-growth.
    For the second case, we allow $V$ to have singularity at $0$ with order no more than $2$. When $q=2$, the potential $V(x)=w(\frac{x}{\av{x}})/|x|^2$ does not need to be non-negative. The importance of the case $q=2$ is explained in \cite{Poon1996}.
\end{remark}

The paper is structured as follows. In Section \ref{sec2}, we review the essential background on Gaussian measures and define the parabolic frequency on Gaussian spaces with the backward Mehler kernel. Section \ref{sec3} is devoted to the almost-monotonicity property of the parabolic frequency. The proofs of our main unique continuation theorems are completed in Section \ref{sec4}.

\section{Parabolic Frequency on Gaussian Spaces}\label{sec2}
In this section, we establish the notation and recall fundamental properties of the Gaussian measure and the Mehler kernel that will be essential for our analysis. Based on this foundation, we develop the theory of parabolic frequency on Gaussian spaces. 

The Ornstein-Uhlenbeck operator $L_\gamma=\nabla^*\nabla$ is self-adjoint with respect to the Gaussian measure, where $\nabla^*=\nabla-x$. Indeed, for $\phi, \psi \in W^{1,2}\bracket{\R^n;d\gamma}$, integration by parts yields
$$
\int_{\R^n} \phi L_\gamma \psi \, d\gamma = -\int_{\R^n}  \nabla \phi\cdot \nabla \psi  \, d\gamma = \int_{\R^n} \psi L_\gamma \phi \, d\gamma.
$$

The Mehler kernel, which serves as the fundamental solution for the Ornstein-Uhlenbeck heat equation, is given by the following explicit formula:
\begin{equation}
    M_H(x,y,t):=\frac{1}{\bracket{1-e^{-2t}}^{{n}/{2}}}\exp\bracket{-\frac{1}{2}\frac{e^{-2t}(\av{x}^2+\av{y}^2)-2e^{-t}x\cdot y}{1-e^{-2t}}}.
\end{equation}

A fundamental property of the Mehler kernel is that it preserves the total Gaussian measure: for every fixed $t > 0$,
\begin{equation}
    \int_{\R^n} M_H(x,y,t)d\gamma(y) = 1.
\end{equation}

For our analysis, we employ the backward Mehler kernel, which is obtained by time reversal and choosing $y=0$. The backward Mehler kernel starting from the point $(x_0,t_0)$ is given by
\begin{equation}\label{E:backward_Mehler_Kernel_x_t}
    M_{x_0,t_0}(x,t):=\frac{1}{\bracket{1-e^{-2(t_0-t)}}^{{n}/{2}}}\exp\bracket{-\frac{1}{2}\frac{e^{-2(t_0-t)}(\av{x-x_0}^2)}{1-e^{-2(t_0-t)}}}.
\end{equation}

Throughout this paper, we focus on the case where $(x_0,t_0) = (0, 0)$ with $0\in\Omega$, which simplifies the backward Mehler kernel to
\begin{equation}\label{backwardMehlerkernel}
    M(x,t):=\frac{1}{\bracket{1-e^{2t}}^{-{n}/{2}}}\exp\bracket{-\frac{1}{2}\frac{e^{2t}\av{x}^2}{1-e^{2t}}}.
\end{equation}

The backward Mehler kernel satisfies the backward heat equation with respect to the Ornstein-Uhlenbeck operator:
\begin{equation}\label{backward_equation}
    \partial_t M = -L_\gamma M.
\end{equation}
This property will be crucial in deriving the monotonicity formula for the parabolic frequency. Note that there is a matrix equality for the backward Mehler kernel $M$.
\begin{equation}\label{E:matrix Harnack estimate}
    \nabla_i \nabla_j M - \frac{\nabla_i M \nabla_j M}{M} +  \frac{e^{2t}}{1-e^{2t}}M g_{ij}= 0,
\end{equation}
which differs from Hamilton's matrix Harnack estimate \cite{Hamilton1993} by the factor $1/2t$ replaced by ${e^{2t}}/({1-e^{2t}})$.

Following the approach pioneered by Poon \cite{Poon1996} and adapted to our Gaussian setting, we consider the weighted $L^2$ integral of $u$ with respect to the backward Mehler kernel. Let $\Omega\subset\R^n$ be a convex domain or $\Omega=\R^n$. Define the weighted $L^2$ integral of $u$ by
\begin{equation}\label{equation_H}
    H(\tau):=\int_{\Omega} u^2(x,-\tau)M(x,-\tau)d\gamma(x)
\end{equation}
and the weighted $L^2$ integral of $\nabla u$ by
\begin{equation}\label{equation_I}
    I(\tau):=\int_{\Omega} \av{\nabla u}^2(x,-\tau)M(x,-\tau)d\gamma(x)
\end{equation}
for $0 < \tau < T$.
Then we define the parabolic frequency on Gaussian spaces by
\begin{equation}\label{equation_N}
    N(\tau) := \bracket{1-e^{-2\tau}}\frac{I(\tau)}{H(\tau)}.
\end{equation}

\begin{remark}
    Notably, Colding and Minicozzi \cite{ColdingMinicozzi2022} explicitly observe that their general framework, when specialized to Euclidean space with Ornstein-Uhlenbeck drift, directly implies Poon's classical frequency monotonicity. They use the Euclidean parabolic frequency
    \begin{equation*}
        N(t) := \frac{\int_\Omega \av{\nabla u}^2d\gamma}{\int_\Omega u^2d\gamma}.
    \end{equation*}
    Although the drift Laplacian framework achieves remarkable universality applicable to all manifolds, it necessarily sacrifices the rich geometric structures intrinsic to Gaussian spaces. For example, Lemma \ref{L:vanishing_order_of_H(R)} will no longer hold without the Mehler kernel weight. Moreover, with this definition, one can only obtain backward uniqueness rather than strong unique continuation. In addition, to make integration by parts valid, the class of solutions $u$ of \eqref{heat_equation} considered in \cite{ColdingMinicozzi2022} is more restrictive than the space of functions with $G(A,B)$-growth. Thus, the Mehler kernel weighted integral represents a more natural and canonical choice in this setting.
\end{remark}

Here, we give an elementary inequality for the parabolic frequency on Gaussian spaces.

\begin{lemma}\label{L:Calculation_of_frequency}
    Assume that $u$ is a solution of \eqref{nonlinear_equation} with $u\in  W^{2,2}\bracket{\Omega;e^{-2A\av{x}^2}d\gamma}$. Let $T_0=\frac{1}{2}\min\{\log\bracket{1+\frac{1}{8(A+\pi)}},T\}$.  Then  for any positive $\tau<T_0$, the derivative of $N(\tau)$ satisfies
    \begin{equation}\label{E:N'}
        N^\prime(\tau)\geq-\frac{1-e^{-2\tau}}{2H(\tau)}\int_{\Omega} F(u,\nabla u)^2 M d\gamma.
    \end{equation}
\end{lemma}
\begin{proof}
For any $\tau<T_0$, we have
\begin{align*}
    \exp\bracket{2A\av{x}^2-\frac{1}{2}\frac{e^{-2\tau}\av{x}^2}{1-e^{-2\tau}}}\leq 1. 
\end{align*}
Together with the assumption that $u\in  W^{2,2}\bracket{\Omega;e^{-2A\av{x}^2}d\gamma}$, it follows that  $H(\tau),I(\tau)$ and $N(\tau)$ are well-defined and subsequent integration by parts is justified.
   
   If $\Omega=\R^n$, then together with equations \eqref{nonlinear_equation}, we obtain
    \begin{align}\label{E:equation_I_computation}
        I(\tau)&=-\int_{\R^n} u\bracket{L_\gamma u+\nabla u\cdot\frac{\nabla M}{M}}Md\gamma\\
        &=-\int_{\R^n} u\bracket{u_t-F(u,\nabla u)+\nabla u\cdot\frac{\nabla M}{M}}Md\gamma\notag.
    \end{align}
    Computing the derivative of $H(\tau)$ with equation \eqref{backward_equation}, we find
    \begin{align}\label{E:equation_H'}
        H^\prime(\tau)&=-\int_{\R^n} \bracket{2uu_tM+u^2M_t}d\gamma\\
        &=-\int_{\R^n} \bracket{2uu_tM-u^2L_\gamma M}d\gamma\notag\\
        &=-2\int_{\R^n} u\bracket{u_t+\nabla u\cdot\frac{\nabla M}{M}}Md\gamma.\notag
    \end{align}
    Similarly, computing $I'(\tau)$ with equations \eqref{nonlinear_equation},  \eqref{backward_equation} and \eqref{E:matrix Harnack estimate}, we derive
    \begin{align*}
        I^\prime(\tau)&=-\int_{\R^n} (2\nabla u \cdot \nabla u_t M + |\nabla u|^2 M_t) d\gamma \notag\\
        &=-\int_{\R^n} (2\nabla u \cdot \nabla u_t M - |\nabla u|^2 L_\gamma M) d\gamma \notag\\
        &=2\int_{\R^n} \bracket{u_t\bracket{L_\gamma u+\nabla u\cdot\frac{\nabla M}{M}}M - \nabla_jM\nabla_j\nabla_i u\nabla_i u} d\gamma \notag\\
        &=2\int_{\R^n} u_t\bracket{L_\gamma u+\nabla u\cdot\frac{\nabla M}{M}}Md\gamma+2\int_{\R^n} \nabla_i\nabla_j M\nabla_i u\nabla_ju d\gamma\\
        &\quad+2\int_{\R^n} L_\gamma u \nabla u \cdot \nabla Md\gamma\\
         &=2\int_{\R^n} u_t\left(u_t-F+\nabla u\cdot\frac{\nabla M}{M}\right)Md\gamma+2\int_{\R^n}\frac{(\nabla u\cdot \nabla M)^2}{M} d\gamma\\
        &\quad-2\frac{e^{-2\tau}}{1-e^{-2\tau}}I(\tau)+2\int_{\R^n} \left(u_t-F\right) \nabla u \cdot \nabla Md\gamma\\
        &=2\int_{\R^n} \left(u_t+\frac{{\nabla u \cdot \nabla M}}{M}-\frac{1}{2}F\right)^2 Md\gamma-\frac{1}{2}\int_{\R^n}F^2 Md\gamma-2\frac{e^{-2\tau}}{1-e^{-2\tau}}I(\tau).
    \end{align*}
    Note that 
    \begin{align*}
        I(\tau)H^\prime(\tau)=&2\left(\int_{\R^n} u\big(u_t+\nabla u\cdot\frac{\nabla M}{M}-\frac{1}{2}F\big)Md\gamma\right)^2
         -\frac{1}{2}\left(\int_{\R^n} FuMd\gamma\right)^2,
    \end{align*}
    therefore, we compute   
    \begin{align*}
       \frac{H(\tau)^2}{1-e^{-2\tau}} N^\prime(\tau)&=H(\tau)I^\prime(\tau)-I(\tau)H^\prime(\tau)+\frac{2e^{-2\tau}}{1-e^{-2\tau}}I(\tau)H(\tau)\\
        &=2\int_{\R^n} \bracket{{u_t}+\frac{{\nabla u\cdot\nabla M}}{M}-\frac{1}{2}F}^2Md\gamma\int_{\R^n} u^2Md\gamma.\\
        &-2\left(\int_{\R^n} u\big(u_t+\nabla u\cdot\frac{\nabla M}{M}-\frac{1}{2}F\big)Md\gamma\right)^2\\
        &-\frac{1}{2}\int_{\R^n}F^2Md\gamma\int_{\R^n} u^2Md\gamma+\frac{1}{2}\bracket{\int_{\R^n} FuMd\gamma}^2\\
        &\geq -\frac{1}{2}\int_{\R^n}F^2Md\gamma \int_{\R^n} u^2Md\gamma+\frac{1}{2}\bracket{\int_{\R^n} FuMd\gamma}^2,
    \end{align*}
    where we used  H\"{o}lder's inequality in the inequality. Hence we have
    \begin{align*}
       \frac{H(\tau)^2}{1-e^{-2\tau}} N^\prime(\tau)\ge   -\frac{1}{2}H(\tau) \int_{\R^n}F^2Md\gamma, 
    \end{align*}
proving \eqref{E:N'}.

 If $\Omega$ is a convex domain with boundary, let $d\gamma_\sigma:=e^{-\frac{\av{x}^2}{2}}d\sigma$ denote the Gaussian-weighted surface measure on $\partial\Omega$ and $\nu$ denote the unit outward normal on $\partial\Omega$. Then similar calculations yield
 \begin{align*}
    I(\tau)=-\int_{\Omega} u\bracket{u_t-F+\nabla u\cdot\frac{\nabla M}{M}}Md\gamma+\int_{\partial\Omega}u\partiald{u}{\nu} Md\gamma_\sigma, 
 \end{align*}
 and 
 \begin{align*}
     H^\prime(\tau)=-2\int_{\Omega} u\bracket{u_t+\nabla u\cdot\frac{\nabla M}{M}}Md\gamma+\int_{\partial\Omega}u^2\partiald{M}{\nu}d\gamma_\sigma,
 \end{align*}
 and
    \begin{align*}
       I^\prime(\tau)=&2\int_{\Omega} \left(\bracket{u_t+\frac{{\nabla u \cdot \nabla M}}{M}-\frac{1}{2}F}^2-\frac{1}{4}F^2\right)M d\gamma-\frac{2e^{-2\tau}}{1-e^{-2\tau}}I(\tau) \\
        &+\int_{\partial\Omega}\bracket{\av{\nabla u}^2\partiald{M}{\nu}-2u_t\partiald{u}{\nu} M-2\nabla u\cdot{\nabla M}\partiald{u}{\nu}}d\gamma_\sigma.
    \end{align*}
    Since $u(x,t)=0$ for all $x\in\partial\Omega$, we know that $u_t(x,t)=0$ and ${\nabla u(x,t)}=(\partial u/\partial \nu)\nu$  for all $x\in\partial\Omega$. Thus,
    \begin{align*}
        \int_{\partial\Omega}\bracket{\av{\nabla u}^2\partiald{M}{\nu}-2u_t\partiald{u}{\nu} M-2\nabla u\cdot{\nabla M}\partiald{u}{\nu}}d\gamma_\sigma=\int_{\partial\Omega}\av{\nabla u}^2\frac{(x\cdot\nu)e^{-2\tau}}{1-e^{-2\tau}}Md\gamma_\sigma,
    \end{align*}
    where we used the fact
    \begin{equation*}
        \partiald{M}{\nu}=\nabla M\cdot\nu=-\frac{(x\cdot\nu)e^{-2\tau}}{1-e^{-2\tau}}M.
    \end{equation*}
    Since $\Omega$ is convex, we obtain $x\cdot\nu\geq 0$ on $\partial\Omega$. Thus, \eqref{E:N'} also holds when $\Omega$ is a convex domain with boundary.
\end{proof}

Now it is easy to see that Theorem \ref{T:theorem_1} holds as a corollary of Lemma \ref{L:Calculation_of_frequency}.
\begin{proof}[Proof of Theorem \ref{T:theorem_1}]
    For the solution of \eqref{heat_equation}, $F(u,\nabla u)=0$. Thus, the inequality \eqref{E:N'} gives the monotonicity of $N(\tau)$.
\end{proof}
\begin{remark}
    For any $s>-T$, we define the parabolic frequency starting from  $t = s$ by
    \begin{equation}\label{frequency_s}
        N_s(\tau) := \bracket{1-e^{-2\tau}}\frac{\int_{\Omega} \av{\nabla u}^2(x,s-\tau)M(x,-\tau)d\gamma(x)}{\int_{\Omega} {u}^2(x,s-\tau)M(x,-\tau)d\gamma(x)}.
    \end{equation}
    The same calculation gives that if $u$ satisfies the assumption of Theorem \ref{T:theorem_1}, then for any $s>-T$, $N_s(\tau)$ increases monotonically in the interval $(0,\min\{T_0(A),T+s\})$.
\end{remark}

\section{Almost-monotonicity of Parabolic Frequency}\label{sec3}

In this section, we establish the almost-monotonicity property of the parabolic frequency.

First, we have the following lemma, which is used to handle the growth and singularity of the coefficient $c$.  Inequality \eqref{E:singularity_at_a}, a Hardy-type inequality, is an analog of Poon \cite{Poon1996}*{Proposition 3.1} in the setting of Gaussian spaces.
\begin{lemma}\label{L:Linear_growth_and_Singularity}
    Assume that $u:\Omega\times(-T,+\infty)\rightarrow\R$ satisfies $u\in  W^{1,2}\bracket{\Omega;e^{-2A\av{x}^2}d\gamma}$ for every $t>-T$ and $u(x,t)=0$ for any $(x,t)\in\partial\Omega\times(-T,+\infty)$. Then for any $0<\tau<T$, 
    \begin{equation}\label{E:quadratic_growth}
        \int_{\Omega\times\{t=-\tau\}} \av{x}^2u^2Md\gamma
        \leq\frac{1-e^{-2\tau}}{e^{-2\tau}}\bracket{n\int_{\Omega\times\{t=-\tau\}} u^2Md\gamma+\int_{\Omega\times\{t=-\tau\}} \av{\nabla u}^2Md\gamma}.
    \end{equation}
    Moreover, when $n\geq 3$, for any  $0<\tau<T$, 
    \begin{align}\label{E:singularity_at_a}
        \int_{\Omega\times\{t=-\tau\}} \frac{u^2}{\av{x}^2}Md\gamma
        \leq & \frac{2}{(n-2)(1-e^{-2\tau})}\int_{\Omega} u^2Md\gamma+\frac
        {4}{(n-2)^2}\int_{\Omega} \av{\nabla u}^2Md\gamma.
    \end{align}
\end{lemma}
\begin{proof}
Using the boundary condition that $u(x,t)=0$ for any $x\in\partial\Omega\times(-T,+\infty)$ and the expression of $M$ given in \eqref{backwardMehlerkernel},  for any $\tau>0$  we have 
    \begin{align*}
        \int_{\Omega\times\{t=-\tau\}} \av{x}^2u^2Md\gamma 
        &=  -\frac{1-e^{-2\tau}}{e^{-2\tau}}\int_{\Omega\times\{t=-\tau\}}u^2x\cdot\nabla Md\gamma\\
        &= \frac{1-e^{-2\tau}}{e^{-2\tau}}{\int_{\Omega\times\{t=-\tau\}} u(2x\cdot\nabla u+nu-\av{x}^2u)Md\gamma}\\
        &\leq  \frac{1-e^{-2\tau}}{e^{-2\tau}}\bracket{n\int_{\Omega\times\{t=-\tau\}} u^2Md\gamma+\int_{\Omega\times\{t=-\tau\}} \av{\nabla u}^2Md\gamma},
    \end{align*}
    where in the above inequality we used the fact  $$2\av{\bracket{x\cdot\nabla u} u}\le |x|^2u^2+|\nabla u|^2.$$
    
    Similarly, combining with equation \eqref{backwardMehlerkernel}, when $n\geq 3$, for any $\tau>0$, we have
    \begin{align*}
        \int_{\Omega\times\{t=-\tau\}} {u^2}Md\gamma
        &= \int_{\Omega\times\{t=-\tau\}} \frac{u^2}{\av{x}^2}x\cdot\bracket{-\frac{1-e^{-2\tau}}{e^{-2\tau}}\nabla M}d\gamma\\
        &=\frac{1-e^{-2\tau}}{e^{-2\tau}}\int_{\Omega\times\{t=-\tau\}} \bracket{\frac{(n-2)u^2}{\av{x}^2}+2u\nabla u\cdot\frac{x}{\av{x}^2}-u^2}Md\gamma\\
        &\geq  \frac{1-e^{-2\tau}}{e^{-2\tau}}\int_{\Omega\times\{t=-\tau\}} \bracket{\frac{(n-2)u^2}{2\av{x}^2}-\frac{2}{n-2}\av{\nabla u}^2-u^2}Md\gamma,
    \end{align*}
    which implies the inequality \eqref{E:singularity_at_a}.
\end{proof}

The following lemma establishes the key almost-monotonicity property of the parabolic frequency. 

\begin{lemma}\label{L:estimated_N(R)}
    Assume that $u$ is a solution of \eqref{heat_equation} with $u\in  W^{2,2}\bracket{\Omega;e^{-2A\av{x}^2}d\gamma}$. Assume that $b$ and $c$ satisfy \eqref{b_and_c_bounded}. Let $T_0$ be the constant given by Lemma  \ref{L:Calculation_of_frequency}. Then for any $0<\tau<\tau_0<T_0/2$,
    \begin{equation}\label{estimated_N(R)}
        N(\tau)+1 \leq C\bracket{N(\tau_0)+1},
    \end{equation}
    where the constant C depends on $L$ and $n$.
\end{lemma}
\begin{proof}
Note that for any $0<\tau<T_0/2<1/8$, $e^{2\tau}<2$. Since $F(u,\nabla u)=b\cdot\nabla u+cu$, with the inequality \eqref{b_and_c_bounded} and Lemma \ref{L:Linear_growth_and_Singularity}, we have
\begin{align}\label{E:F(u,gradient_u)_estimate}
    \int_{\Omega\times\{t=-\tau\}} F(u,\nabla u)^2Md\gamma
    &\leq 2\int_{\Omega\times\{t=-\tau\}} \av{b}^2\av{\nabla u}^2Md\gamma+2\int_{\Omega\times\{t=-\tau\}} \av{c}^2\av{u}^2Md\gamma\\
    &\leq 2L^2 I(\tau) + 2L^2\int_{\Omega} \bracket{1+\av{x}^2}\av{u}^2Md\gamma\notag\\
    &\leq 2(n+1)L^2\bracket{I(\tau)+{H(\tau)}}\notag.
\end{align}
Combining the inequality \eqref{E:N'}, we obtain
    \begin{equation*}
        N^\prime(\tau)
        \geq-\frac{1-e^{-2\tau}}{2H(\tau)}\int_{\Omega} F(u,\nabla u)^2 M d\gamma
        \geq-C\bracket{N(\tau)+1},
    \end{equation*}
    where the constant $C=(n+1)L^2$. 
    Integrating this differential inequality gives
    \begin{equation*}
        \log\bracket{\frac{N(\tau_0)+1}{N(\tau)+1}}\geq -C,
    \end{equation*}
    which  yields the desired result \eqref{estimated_N(R)}.
\end{proof}

\section{Unique Continuation for Parabolic Equations}\label{sec4}

In this section, we establish our main unique continuation results.
The following lemma provides a crucial polynomial growth estimate for the term $H(\tau)$ of solutions that vanish to high order, generalizing the classical result of Poon \cite{Poon1996}  to the Gaussian setting.
\begin{lemma}\label{L:vanishing_order_of_H(R)}
    Let $v:\Omega\times(-T,+\infty)$ be a function with $G(A,B(t))$-growth for some constant $A>0$ and some continuous function $B(t)\in C((-T,+\infty))$. Suppose that there exist a positive constant $T_1<T$ and an integer $K > 0$ such that
    \begin{equation*}
        \av{v(x,t)}\leq C_0\bracket{\av{x}^{2K}+\av{t}^{K}},
    \end{equation*}
  whenever $\av{x}^{2}+\av{t}\leq T_1$. 
    Then there exists a constant $T_0=\frac{1}{2}\min\{\log\bracket{1+\frac{1}{8(A+\pi)}},T_1\}$, such that for all positive $\tau<T_0$,
    \begin{equation}
        \int_{\Omega}v^2(x,-\tau)M(x,-\tau)d\gamma(x)\leq C_1\tau^{2K},
    \end{equation}
where the constant $C_1$  depends on $n$, $K$, $\aV{B}_{L^{\infty}([-T_0,0])}$ and $C_0$.

\end{lemma}
\begin{proof}
    Since $B(t)$ is a continuous function, we may assume that $B(t)\leq B_0$ on $[-T_0,0]$ for some constant $B_0>0$ and $v$ has $G(A,B_0)$-growth.
    Thus, with a direct calculation, we have for any $\tau<T_0$,
\begin{align}\label{E:G(A,B)controlled}
    &\aV{v^2(-\tau,x)\exp\bracket{-\frac{\av{x}^2}{4\bracket{1-e^{-2\tau}}}}}_{L^1(\Omega)}\\
    \leq&\quad B_0^2\aV{\exp\bracket{2A\av{x}^2-\frac{\av{x}^2}{4\bracket{1-e^{-2\tau}}}}}_{L^1(\Omega)}\notag\\
    \leq &\quad B_0^2\int_{\R^n} e^{-\pi\av{x}^2}dx = B_0^2.\notag
\end{align}
    We decompose the integral into two regions:
    \begin{align*}
        &\int_{\Omega}v^2(x,-\tau)M(x,-\tau)d\gamma \\
        =& \int_{\Omega\cap\{\av{x}\leq \tau^{1/4}\}}v^2(x,-\tau)M(x,-\tau)d\gamma + \int_{\Omega\cap\{\av{x}>\tau^{1/4}\}}v^2(x,-\tau)M(x,-\tau)d\gamma.
    \end{align*}
For the first integral, using the vanishing assumption, we have
    \begin{equation*}
        \int_{\Omega\cap\{\av{x}\leq \tau^{1/4}\}}v^2(x,-\tau)M(x,-\tau)d\gamma \leq 2C_0^2\tau^{2K}\int_{\Omega}M(x,-\tau)d\gamma \leq  2C_0^2 \tau^{2K}.
    \end{equation*}  
    For the second integral, using the explicit form \eqref{backwardMehlerkernel} of $M$ and the inequality \eqref{E:G(A,B)controlled}, we estimate that
    \begin{align*}
        &\int_{\Omega\cap\{\av{x}\geq \tau^{1/4}\}}v^2(x,-\tau)M(x,-\tau)d\gamma \\
        \leq &\bracket{1-e^{-2\tau}}^{-n/2}\exp\bracket{-\frac{\tau^{1/2} e^{-2\tau}}{4\bracket{1-e^{-2\tau}}}}\aV{v^2(-\tau,x)\exp\bracket{-\frac{\av{x}^2}{4\bracket{1-e^{-2\tau}}}}}_{L^1(\Omega)}\\
        \leq & B_0^2\bracket{1-e^{-2\tau}}^{-n/2}\exp\bracket{-\frac{\tau^{1/2} e^{-2\tau}}{4\bracket{1-e^{-2\tau}}}}.
    \end{align*}
For $\tau < T_0 < 1/8$, we have $\tau \leq 1 - e^{-2\tau} \leq 2\tau$ and $e^{-2\tau} > 1/2$. Therefore,
    \begin{equation*}
        \bracket{1-e^{-2\tau}}^{-n/2}\exp\bracket{-\frac{\tau^{1/2} e^{-2\tau}}{4\bracket{1-e^{-2\tau}}}}\leq \frac{1}{\tau^{n/2}} e^{-\frac{1}{16\tau^{1/2}}}.
    \end{equation*}
    
    Consider the function $g(\theta) = -(n + 4K)\log \theta - 1/(16\theta)$. By  calculus, we find that $g(\theta)$ achieves its maximum at $\theta = 1/(16(n + 4K))$, resulting in
    \begin{equation*}
        \frac{1}{\tau^{n/2}} e^{-\frac{1}{16\tau^{1/2}}}=\tau^{2K}e^{g(\tau^{1/2})} \leq C_{n,K}\tau^{2K},
    \end{equation*}
    where the constant $C_{n,K}$ is given by $\log C_{n,K}:=g({1}/{16(n + 4K)})$.
    Thus, 
    \begin{equation*}
        \bracket{1-e^{-2\tau}}^{-n/2}\exp\bracket{-\frac{\tau^{1/2} e^{-2\tau}}{2\bracket{1-e^{-2\tau}}}}\leq {C_{n,K}}\tau^{2K}.
    \end{equation*}
    Combining both estimates, we conclude that
    \begin{align*}
        \int_{\Omega}v^2(x,-\tau)M(x,-\tau)d\gamma
        \leq\bracket{2C_0^2+C_{n,K}B_0^2}\tau^{2K}.
    \end{align*}
    This completes the proof with $C_1:=2C_0^2+C_{n,K}B_0^2$.
\end{proof}

Now we are ready to prove our main theorem with a standard argument.

\begin{proof}[Proof of Theorem \ref{main_theorem}]
By translation, we may assume without loss of generality that the solution $u$ vanishes to infinite order at the origin $(0, 0)$ and the conditions of \eqref{b_and_c_bounded} are preserved. We argue by contradiction. Assume that $u$ is a nontrivial solution. Since $u(x,t)$ has $G(A,B(t))$-growth and $\av{\nabla u}\in  W^{1,2}\bracket{\Omega;e^{-2A\av{x}^2}d\gamma}$, we obtain that $u\in  W^{2,2}\bracket{\Omega;e^{-2A\av{x}^2}d\gamma}$.

First, we claim that $u(x,t)\equiv 0$ for all $x\in\Omega$ and $t\leq 0$. If not, define
\begin{equation*}
    s_0:=\inf\{s:u(x,t)=0\ \textit{for all}\ x\in\R^n\ \textit{and}\ s\leq t\leq 0\}>-T.
\end{equation*}
Without loss of generality, we may assume $s_0=0$, otherwise we can consider $\tilde{u}(x,t)={u}(x,t-s_0)$.
Thus, there is some positive constant $S<T_0/2$ small enough, such that $H(\tau)>0$ for all positive $\tau<S$, where $T_0$ is given in Lemma \ref{L:Calculation_of_frequency}.

Note from inequalities \eqref{E:equation_I_computation}, \eqref{E:equation_H'} and \eqref{E:F(u,gradient_u)_estimate} that 
\begin{align*}
    H^\prime(\tau) = 2\bracket{I(\tau)-\int_{\Omega} u(b \cdot \nabla u + cu) M d\gamma}\leq 2(n+2)L^2\bracket{I(\tau)+{H(\tau)}}.
\end{align*}
Therefore, from inequality \eqref{estimated_N(R)}, we obtain for any $0<\tau<\tau_0<S$,
\begin{equation*}
    \frac{(1-e^{-2\tau})H^\prime(\tau)}{H(\tau)}\leq 2(n+2)L^2\bracket{N(\tau)+1}\leq C_2e^{-2\tau},
\end{equation*}
where $C_2$ is a constant depending on $n,L,N(S)$.
Integrating this inequality from $\tau$ to $\tau_0$ yields
\begin{align*}
    \log\frac{H(\tau_0)}{H(\tau)}\leq C_2\log\bracket{\frac{1-e^{-2\tau_0}}{1-e^{-2\tau}}}.
\end{align*}
Therefore, for any $0<\tau<\tau_0$,
\begin{align}\label{E:growth_of_H}
    H(\tau)\geq \frac{ H(\tau_0)}{\bracket{1-e^{-2\tau_0}}^{C_2}}\bracket{1-e^{-2\tau}}^{C_2}\geq \frac{ H(\tau_0)}{\bracket{1-e^{-2\tau_0}}^{C_2}}\tau^{C_2}.
\end{align}

However, Lemma \ref{L:vanishing_order_of_H(R)} shows that if $u$ vanishes to infinite order at $(0,0)$, then for any $K>0$, there is some constant $C_1=C(n,K,A,C_0,\aV{B}_{L^{\infty}([-T_0,0])})$ such that 
\begin{equation*}
    H(\tau)\leq C_1\tau^{2K}.
\end{equation*}
This leads to a contradiction with \eqref{E:growth_of_H} when we choose $K>C_2/2$.

Now it suffices to show that $u(x,t)\equiv 0$ for all $x\in\Omega$ and $t>0$. Similarly, if not, after translation in time, we may assume that there is some constant $0<S<T_0/2$ such that $H_{2S}(\tau)\neq 0$ for any $0<\tau<2S$, where 
\begin{equation*}
    H_{2S}(\tau):=\int_{\Omega} {u}^2(x,2S-\tau)M(x,-\tau)d\gamma(x).
\end{equation*}
Let $N_{2S}(\tau)$ be given by \eqref{frequency_s}. Then with inequalities \eqref{E:equation_I_computation}, \eqref{E:equation_H'} and \eqref{E:F(u,gradient_u)_estimate}, we derive
\begin{align*}
    H_{2S}^\prime(\tau) = 2\bracket{I_{2S}(\tau)-\int_{\Omega} u(b \cdot \nabla u + cu) M d\gamma}\geq -2(n+2)L^2\bracket{I_{2S}(\tau)+{H_{2S}(\tau)}}.
\end{align*}
Thus, it follows from inequality \eqref{estimated_N(R)} that for any $S<\tau<2S$,
\begin{equation*}
    \frac{(1-e^{-2\tau})H_{2S}^\prime(\tau)}{H_{2S}(\tau)}\geq -2(n+2)L^2\bracket{N_{2S}(\tau)+1}\geq -C_2e^{-2\tau},
\end{equation*}
where $C_2$ is a constant depending on $n,L,N_{2S}(S)$.
Integrating this inequality from $S$ to $\tau$ yields
\begin{align*}
    \log\frac{H_{2S}(\tau)}{H_{2S}(S)}\geq -C_2\log\bracket{\frac{1-e^{-2\tau}}{1-e^{-2S}}}.
\end{align*}
Therefore, for any $S<\tau<2S$,
\begin{align*}
    H_{2S}(\tau)\geq \frac{ H_{2S}(S)}{\bracket{1-e^{-2\tau}}^{C_2}}\bracket{1-e^{-2S}}^{C_2}.
\end{align*}
Letting $\tau\rightarrow 2S$ yields
\begin{align*}
    0=H_{2S}(2S)\geq\frac{ H_{2S}(S)}{\bracket{1-e^{-4S}}^{C_2}}\bracket{1-e^{-2S}}^{C_2}>0,
\end{align*}
which is a contradiction. Hence,  $u$ is identically zero.
\end{proof}
\begin{remark}
    From the above argument and Lemma \ref{L:Calculation_of_frequency}, we know that the conclusion of Theorem \ref{main_theorem} also holds for solutions of the nonlinear parabolic equation  \eqref{nonlinear_equation} with the nonlinear term $F(u,\nabla u)$ satisfying that for some constant $L>0$,
    \begin{equation*}
        \av{F(u,\nabla u)}\leq L\left(\av{\nabla u}+(1+\av{x})\av{u}\right).
    \end{equation*}
 \end{remark}

To derive the unique continuation for \eqref{heat_equation_potential} with a singular potential $V(x)$, we need a slightly different definition.
Let $u$ be a solution of \eqref{heat_equation_potential}. Define
\begin{align*}
    H(\tau):=\int_{\Omega\times\{t=-\tau\}} u^2Md\gamma, \ I(\tau):=\int_{\Omega\times\{t=-\tau\}} \bracket{\av{\nabla u}^2+Vu^2}Md\gamma
\end{align*}
and 
\begin{align*}
    N(\tau) := \bracket{1-e^{-2\tau}}\frac{I(\tau)}{H(\tau)}.
\end{align*}
With a similar calculation as in the proof of Lemma \ref{L:Calculation_of_frequency},  we have
\begin{align*}
      H^\prime(\tau)=-2\int_{\Omega} u\bracket{u_t+\nabla u\cdot\frac{\nabla M}{M}}Md\gamma=2I(\tau),
\end{align*}
and
\begin{align*}
    I^\prime(\tau)=&2\int_{\Omega} \left(\bracket{u_t+\frac{{\nabla u \cdot \nabla M}}{M}}^2-\frac{e^{-2\tau}}{1-e^{-2\tau}}\av{\nabla u}^2\right)M d\gamma\\
        &+\frac{e^{-2\tau}}{1-e^{-2\tau}}\int_{\Omega}\bracket{x\cdot\nabla V}u^2M d\gamma+\int_{\partial\Omega}\bracket{\av{\nabla u}^2\partiald{M}{\nu}-2\nabla u\cdot{\nabla M}\partiald{u}{\nu}}d\gamma_\sigma.
\end{align*}
Since $\Omega$ is convex, we know that
\begin{align*}
        \int_{\partial\Omega}\bracket{\av{\nabla u}^2\partiald{M}{\nu}-2\nabla u\cdot{\nabla M}\partiald{u}{\nu}}d\gamma_\sigma=\int_{\partial\Omega}\av{\nabla u}^2\frac{(x\cdot\nu)e^{-2\tau}}{1-e^{-2\tau}}Md\gamma_\sigma\geq 0.
    \end{align*}
    Therefore, applying H\"{o}lder's inequality, we have 
\begin{align}\label{E:N'(tau)_V}
        N^\prime(\tau)&= \frac{1-e^{-2\tau}}{H(\tau)^2}\bracket{H(\tau)I^\prime(\tau)-I(\tau)H^\prime(\tau)+\frac{2e^{-2\tau}}{1-e^{-2\tau}}I(\tau)H(\tau)}\\
        &\geq \frac{e^{-2\tau}}{H(\tau)}\int_{\Omega\times\{t=-\tau\}}\bracket{x\cdot\nabla V+2V}u^2M d\gamma.\notag
    \end{align}
Following the approach presented in  \cite{Poon1996}*{Proof of Theorem 1.2}, we prove Theorem \ref{T:main_theorem_V}.
\begin{proof}[Proof of Theorem \ref{T:main_theorem_V}]
After translation in time, we may assume that the solution $u$ vanishes to infinite order at $(x_0,0)$. 

First, we claim that the point $x_0\neq 0$. We prove this by contradiction: assume that $x_0=0$. 

For the first case \eqref{E:V_case_1}, with variation of parameters, we can derive that there exists a constant $c_0>0$ such that
\begin{equation*}
    \av{v(r)}\leq \frac{r^2}{4}
     + \frac{1}{2} + \frac{c_0}{r^2}. 
\end{equation*}
Since $v$ has no singular points, we may assume $\av{v(r)}\leq 1+r^2$ for $x\in\Omega$ and $r=\av{x}$. 
Note that by the inequality \eqref{E:quadratic_growth}, $I(\tau)$ and $N(\tau)$ are well-defined. Moreover, we obtain
\begin{align*}
    \int_{\Omega\times\{t=-\tau\}}\bracket{1+\av{x}^2}u^2M d\gamma
    &\leq \frac{1-e^{-2\tau}}{e^{-2\tau}}\bracket{nH(\tau)+I(\tau)-\int_{\Omega}Vu^2M d\gamma}\\
    &\leq \frac{1-e^{-2\tau}}{e^{-2\tau}}\bracket{nH(\tau)+I(\tau)-L\int_{\Omega}\bracket{1+\av{x}^2}u^2M d\gamma}.
\end{align*}
Then choosing $T_0$ small enough, we have for any $\tau<T_0$, 
\begin{align*}
    \int_{\Omega\times\{t=-\tau\}}\bracket{1+\av{x}^2}u^2M d\gamma
    \leq 2\frac{1-e^{-2\tau}}{e^{-2\tau}}\bracket{nH(\tau)+I(\tau)}.
\end{align*}
Thus, combining inequalities \eqref{E:V_case_1} and \eqref{E:N'(tau)_V}, it follows that
\begin{align*}
    N^\prime(\tau)&\geq -\frac{e^{-2\tau}L}{H(\tau)}\int_{\Omega\times\{t=-\tau\}}\bracket{1+\av{x}^2}u^2M d\gamma\\
    &\geq -\frac{2L(1-e^{-2\tau})}{H(\tau)}\bracket{nH(\tau)+I(\tau)}\\
    &\geq -C(N(\tau)+1),
\end{align*}
where $C$ is a constant depending on $n,L$.

For the second case, we have
\begin{equation*}
    x\cdot\nabla V+2V =\bracket{rv^\prime(r)+2v(r)} w\bracket{\frac{x}{\av{x}}}=(2-q) V(x).
\end{equation*}
The inequality \eqref{E:singularity_at_a} tells us that
\begin{equation*}
    \int_{\Omega\times\{t=-\tau\}}Vu^2M d\gamma\leq L\int_{\Omega\times\{t=-\tau\}}\bracket{\frac{2-q}{2}+\frac{q}{2\av{x}^2}}u^2M d\gamma<+\infty,
\end{equation*}
which means that $I(\tau)$ and $N(\tau)$ are well-defined.
Since $V\geq 0$ for $q<2$ and $ x\cdot\nabla V+2V=0$ for $q=2$, combining inequality \eqref{E:N'(tau)_V}, we obtain
\begin{align*}
    N^\prime(\tau)&\geq \frac{e^{-2\tau}\bracket{2-q}}{H(\tau)}\int_{\Omega\times\{t=-\tau\}}Vu^2M d\gamma\geq 0.
\end{align*}

Thus, for both cases, $N^\prime(\tau)\geq -C(N(\tau)+1)$, which means the inequality \eqref{estimated_N(R)}
 holds. Thus,
\begin{equation*}
     \frac{(1-e^{-2\tau})H^\prime(\tau)}{H(\tau)}=N(\tau)\leq C_2e^{-2\tau}.
\end{equation*}
Therefore, for any $0<\tau<\tau_0$,
\begin{align*}
    H(\tau)\geq \frac{ H(\tau_0)}{\bracket{1-e^{-2\tau_0}}^{C_2}}\bracket{1-e^{-2\tau}}^{C_2}\geq \frac{ H(\tau_0)}{\bracket{1-e^{-2\tau_0}}^{C_2}}\tau^{C_2},
\end{align*}
which is a contradiction with Lemma \ref{L:vanishing_order_of_H(R)} if $u$ vanishes to infinite order at $(0,0)$. Thus, $x_0\neq 0$.

In both cases, $V\in L_{loc}^{\infty}(\Omega\backslash\{0\})$. It follows from Lin's result \cite{Lin1990} that $u(x,0)\equiv 0$ in $\Omega$. And with a similar argument in the proof of Theorem \ref{main_theorem}, we have $u(x,t)=0$ for all $x\in\Omega$ and $t\geq 0$.

Next, we extend the conclusion  $u(x,t)\equiv 0$ for any $x\in\Omega$ to all times via an induction argument. For a fixed constant $0<S\leq T/2$ and $-2S<\tau<0$, define
\begin{equation*}
    D(\tau):= \int_\Omega u(x,\tau-2S)u(x,-\tau)d\gamma.
\end{equation*}
Then applying integration by parts and noticing that $u_t=L_\gamma u-Vu$ and $u\big|_{\partial\Omega}=0$, we get
\begin{align*}
    D^\prime(\tau)&=\int_\Omega \bracket{u_t(x,\tau-2S)u(x,-\tau)-u(x,\tau-2S)u_t(x,-\tau)}d\gamma\\
    &=\int_\Omega \bracket{L_\gamma u(x,\tau-2S)u(x,-\tau)-u(x,\tau-2S)L_\gamma u(x,-\tau)}d\gamma\\
    &=0.
\end{align*}
Since $D(0)=0$, we obtain that $D(S)=\int_\Omega u(x,-S)^2d\gamma=0$, which implies that $u(x,-S)\equiv0$ for any $x\in\Omega$. Since $S$ is any fixed constant in $(0,T/2]$, we have $u(x,t)\equiv 0$ for any $x\in\Omega$ and $t\geq-T/2$. 

By induction, assume that  $u(x,t)\equiv 0$ for any $x\in\Omega$ and $t\geq-T(1-{2^{-k}})$ for $k\geq 1$. Repeating the argument for $u(x,t+T(1-2^{-k}))$, we derive that $u(x,t)\equiv 0$ for any $x\in\Omega$ and $t\geq-T(1-2^{-k-1})$. Thus,  $u(x,t)\equiv 0$ for any $x\in\Omega$ and $t\geq-T(1-{2^{-k}})$ for any $k\geq 1$. The proof is complete by taking $k\rightarrow\infty$.
\end{proof}

\section*{Acknowledgments}
This work was supported by the Laboratory of Mathematics for Nonlinear Science at Fudan University and the Natural Science Foundation of Jiangsu Province (Grant No. BK20231309).                                                                                                                                                                                                     
\bibliography{Parabolic_Frequency_on_Gaussian_Spaces}                        %refer为.bib文件名

\end{document}